\documentclass{elsarticle}

\usepackage{latexsym}
\usepackage{amsmath}
\usepackage{amssymb}
\usepackage{amsthm}
\usepackage{verbatim}
\usepackage{ifthen}
\usepackage{amscd}
\usepackage{epsfig} %,subfigure}
\usepackage{hyperref}

\numberwithin{equation}{section}
\numberwithin{figure}{section}

\theoremstyle{plain}
\newtheorem{theorem}[equation]{Theorem}
\newtheorem{corollary}[equation]{Corollary}
\newtheorem{lemma}[equation]{Lemma}

\newtheorem{proposition}[equation]{Proposition}

\theoremstyle{definition}

\newtheorem{definition}[equation]{Definition}

\DeclareMathOperator{\Num}{Num}
\DeclareMathOperator{\GL}{GL}

\begin{document}

\bibliographystyle{alpha}

\newcommand{\ZZ}{\mathbb{Z}}
\newcommand{\QQ}{\mathbb{Q}}
\newcommand{\RR}{\mathbb{R}}
\newcommand{\NN}{\mathbb{N}}
\newcommand{\CC}{\mathbb{C}}

\newcommand{\Der}{\ensuremath{\operatorname{Der}}}
\newcommand{\Gr}{\ensuremath{\operatorname{Gr}}}
\newcommand{\id}{\ensuremath{\operatorname{Id}}}
\newcommand{\pder}[1]{\ensuremath{\frac{\partial}{\partial #1}}}
\newcommand{\ddx}[1]{\ensuremath{\frac{d}{d #1}}}

\newcommand{\es}{exact sequence}
\newcommand{\ses}{short \es}
\newcommand{\thcr}{twisted homogeneous coordinate ring}
\newcommand{\nthcr}{twisted homogeneous coordinate ring}%new version of it--change if think of new term
\newcommand{\ann}{\ensuremath{\operatorname{ann}}}
\newcommand{\supp}{\ensuremath{\operatorname{Supp}}}
\newcommand{\im}{\ensuremath{\operatorname{Im}}}
\newcommand{\Hom}{\ensuremath{\operatorname{Hom}}}
\newcommand{\HS}{\ensuremath{\operatorname{HS}}}
\newcommand{\gkdim}{\ensuremath{\operatorname{GKdim}}}
\newcommand{\Kdim}{\ensuremath{\operatorname{Kdim}}}
\newcommand{\Proj}{\ensuremath{\operatorname{Proj}}}
\newcommand{\pr}{\ensuremath{\operatorname{pr}}}
\newcommand{\Pic}{\ensuremath{\operatorname{Pic}}}
\newcommand{\Qgr}{\ensuremath{\operatorname{Qgr}}}
\newcommand{\Tors}{\ensuremath{\operatorname{Tors}}}
\newcommand{\qch}{\ensuremath{\operatorname{qch}}}
\newcommand{\coker}{\ensuremath{\operatorname{coker}}}
\newcommand{\Qch}{\ensuremath{\operatorname{Qch}}}
\newcommand{\GrMod}{\ensuremath{\operatorname{GrMod}}}
\newcommand{\Mod}{\ensuremath{\operatorname{Mod}}}
\newcommand{\Spec}{\ensuremath{\operatorname{Spec}}}
\newcommand{\Char}{\ensuremath{\operatorname{char}}}
\newcommand{\length}{\ensuremath{\operatorname{length}}}
\newcommand{\depth}{\ensuremath{\operatorname{depth}}}
\newcommand{\height}{\ensuremath{\operatorname{height}}}
\newcommand{\dlim}{\ensuremath{\operatorname{\displaystyle \lim_{\rightarrow}}}}
\newcommand{\onto}{\twoheadrightarrow}
\newcommand{\into}{\hookrightarrow}
\newcommand{\mc}[1]{\ensuremath{\mathcal{#1}}}
\newcommand{\mb}[1]{\ensuremath{\mathbb{#1}}}
\newcommand{\Oh}{\ensuremath{\mathcal{O}}}

\begin{frontmatter}

\title{Noncommutative ampleness from finite endomorphisms}

\author{D.~ S.~ Keeler}
\address{Dept. of Mathematics, Miami University, Oxford, OH 45056}
\ead{keelerds@miamioh.edu}

\author{K.~ Retert}
\address{Dept. of Mathematics, Miami University, Oxford, OH 45056}
\ead{retertk@gmail.com}


\begin{abstract}
Let $X$ be a projective integral scheme with endomorphism
$\sigma$, where $\sigma$ is finite, but not an automorphism. 
We examine noncommutative ampleness of bimodules defined by $\sigma$.
In contrast to the automorphism case, one-sided ampleness is possible.
We also find that rings and bimodule algebras associated
with $\sigma$ are not noetherian.
\end{abstract}


\end{frontmatter}

%\linenumbers

\section{Introduction}

Let $X$ be a projective integral scheme over a field $k$. A homogeneous
coordinate ring $R$ can be built from global sections of $\mc{L}^{\otimes n}$,
where $\mc{L}$ is an ample invertible sheaf. The sheaf $\mc{L}$
can be used to prove the Serre Correspondence Theorem, showing
a category equivalence between the tails of finitely generated graded $R$-modules
and coherent sheaves on $X$. 

In \cite{av}, Artin and Van den Bergh 
defined a twisted homogeneous coordinate ring $B$ by using an
automorphism $\sigma$ of $X$ to twist the usual multiplication.
Via a noncommutative definition of ampleness (see Definition~\ref{def: sigma-ample}), 
they again showed
a category equivalence between the tails of graded (right) $B$-modules
and coherent sheaves on $X$.

However, this definition was not clearly equivalent on the left and right.
In \cite{k}, the first author showed that the left and right definitions
are equivalent. A key technique was to study the behavior of $\sigma$
on the numerical equivalence classes of divisors.

In \cite{vdb}, Van den Bergh generalized these definitions to include the
possibility that $\sigma$ was not an automorphism, but only a finite endomorphism.
In this paper, we examine this noncommutative ampleness in the general finite
endomorphism case. Unlike the automorphism case, one never has ampleness
on the left, but it is possible to have ampleness on the right.
In Corollary~\ref{cor: not left ample} and Proposition~\ref{prop: ample eigenvector}
we have (in simplified form)
\begin{theorem}
Let $X$ be a regular projective integral scheme with finite
endomorphism $\sigma$ and invertible sheaf $\mc{L}$. 
Suppose $\sigma$ is not an automorphism. Then
the sequence of $\Oh_X$-bimodules defined by $\sigma$ and $\mc{L}$
is not left ample.

However, if $\mc{L}$ is ample (in the commutative sense)
and $\sigma^*\mc{L} \cong \mc{L}^{\otimes r}$  for some $r \in \ZZ$,
then the sequence is right ample.
\end{theorem}

We also show that in case the sequence is right ample, the resulting
twisted homogeneous coordinate ring is not noetherian, as is
the bimodule algebra which defines the ring. See Theorem~\ref{thm: nonnoetherian ring}
and Corollary~\ref{cor: nonnoetherian bimodule algebra}. 
As a specific example, in Section~\ref{sec:  frobenius example}
we examine the case of $X=\mb{P}^m$ and $\sigma$ the relative Frobenius
endomorphism. We find
\begin{proposition}%\label{prop: Frobenius example}
Let $X=\mb{P}^m$ over a perfect field of characteristic $p > 0$. 
Let $f$ be the relative Frobenius endomorphism. 
Then the twisted homogeneous coordinate ring 
${F = \oplus \Gamma\left(\Oh\left(\frac{p^n-1}{p-1}\right)\right)}$
is not finitely generated, unless $p=2$ and $m=1$. In that case, $F$ is generated
by $F_1 = \Gamma(\Oh(1))$, but $F$ is not noetherian. 
\end{proposition}

The proofs of category equivalences in \cite{av,vdb} rely on noetherian
conditions. Thus we have not generalized those results in this paper.

\section{Finite Endomorphism Properties}

Throughout this paper, $X$ will be an integral scheme of finite-type
over a field $k$. Often we will also assume $X$ is projective, normal,
or regular, in which case we will make these assumptions clear.

We begin by verifying some basic properties of a finite endomorphism
$f$. Presumably these are all well-known.

\begin{lemma}\label{lem:  integral, finite=>inj}
Let $X$ be an integral scheme.  Let $f:X\to X$ be a finite morphism.
Then for any open affine set $U$, the corresponding
map of rings $f|_U ^{\#}$ is injective.  In other words, 
when $\Spec A=f^{-1}(\Spec B)$ we may consider $B$ as a subring
of $A$ (up to isomorphism).
\end{lemma}

\begin{proof}
Since $f$ is finite, 
for any open affine set $U=\Spec B$, the pre-image $f^{-1}(U)=\Spec A$
where $A$ is a $B$-algebra which is finitely generated as a $B$-module.
Now $A$ is a $B$-module via the ring homomorphism $\phi=f|_U ^{\#}$
from $B$ to $A$ corresponding to $f$.

Thus, the map $\phi:B\to A$ makes $A$ integral
over $B$ (technically over $\phi(B)$).  Then by 
\cite[Proposition~9.2]{e}, $\dim A=\dim (B/\ker\phi)$.  
As $X$ is irreducible (being integral),
every nonempty open subset of $X$ is dense in
$X$.  Hence the dimension of every nonempty open subset of $X$ 
equals the dimension of $X$.  In particular, $\dim U=\dim f^{-1}(U)$
and so $\dim B=\dim (B/\ker \phi)$.  Since $X$ is integral, $B$ is
a domain (and $\langle 0\rangle$ is prime) and so the only way to have 
$\dim B=\dim B/\ker \phi$
is for $\ker \phi=\langle 0\rangle$.  Thus $\phi$ must be injective
and we may consider $B$ a subset of $A$.  
\end{proof}

\begin{lemma}\label{lem:  int proj, finite=>surj}
\cite[Theorem~4, p.~61]{sh1}
Let $X$ be an integral scheme with $f$ a finite
endomorphism.  Then $f$ is surjective. \qed
\end{lemma}

\begin{corollary}\label{cor: int proj, finite=>open}
Let $X$ be a normal, integral scheme.  Let $f:X\to X$ be a finite morphism.
Then $f$ is an open morphism.
\end{corollary}
\begin{proof}
By \cite[2.8~Theorem, p.~220]{d},
a finite dominant morphism $f:X\to Y$, where $Y$ is normal, is open.
\end{proof}

The following connection between the degree of $f$ and the possible
invertibility of $\sigma$ will be important for our study of
ampleness. See \cite[p.~299]{kl} for the definition of degree
and its basic properties.

\begin{lemma}\label{lem:  normal, int, finite deg 1=>auto}
Let $X$ be a normal, integral scheme and $f:X\to X$ a finite morphism
of degree~1.  Then $f$ is an automorphism.
\end{lemma}

This fact reflects the description in \cite[p.~219]{d} that ``an algebraic
variety $X$ is said to be \textit{normal} if every finite birational
morphism $X'\to X$ is an isomorphism.''

The following argument is based on the proof of \cite[Proposition~1.2.1.12]{sc}.

\begin{comment}%--for my benefit
\todo{Hartshorne, Ex. III.9.3 says that if $f:X\to Y$ is a 
finite surjective morphism of nonsingular varieties over an
algebraically closed field $k$ then $f$ is flat.}
\end{comment}
\begin{proof}
To show that $f$ is an isomorphism, it suffices to prove that
$f$ is a homeomorphism and that for any open affine $U\subset X$,
we have $f|_{U} ^{\#}$ is an isomorphism of sheaves.
By \cite[Exercise~II.2.18]{h}, the homeomorphism property follows from
the isomorphism of sheaves.

First, since $X$ is reduced and $f$ has degree~1,
the map $f$ must be birational by \cite[p.~299, Example~2]{kl}.

Next, we want that for any open affine $U\subset X$,
the map $f|_{U} ^{\#}$ is an isomorphism of sheaves.  Since $U$ is
affine, we have $U=\Spec B$ for some $k$-algebra $B$.  
Since $f$ is
finite, 
$f^{-1}(U)=\Spec A$ where $A$ is a $B$ algebra which is
a finitely generated $B$-module. 
As in Lemma~\ref{lem:  integral, finite=>inj}, the map 
$f|_{f^{-1}(U)}$ corresponds to some map $\phi:B\to A$ which
induces an inclusion 
$B\subset A$. Since $f$ is birational, we have that
$A$ and $B$ have the same field of fractions.

   As $X$ is normal, we know that 
$B$ is integrally closed; however, since $A$ is a finite $B$-algebra, 
$A$ is integral over $B$.  Thus $A=B$ because $A$ and $B$ have the same
field of fractions. So the sheaf
map induced from $\phi$ is an isomorphism of sheaves.  
\end{proof} 

We must now strengthen the hypothesis on $X$ from
normal to regular. By regular we mean that the stalks of $\Oh_X$
are regular local rings. Recall that regular implies
normal \cite[Exercise~I.5.13]{h}.

\begin{lemma}\label{lem:  normal+finite=>flat}
Let $X$ be a regular, integral scheme and $f:X\to X$ a finite morphism.  
Then $f$ is flat.
\end{lemma}

\begin{proof}  
Since $X$ is regular, $X$ is Cohen-Macaulay \cite[Theorem~II.8.21A]{h}.
Thus $f$ is flat by \cite[Exercise~III.10.9]{h}.
%Thus $f_* \Oh_X$ is locally free by \cite[Corollary~18.17]{e}
%and hence $f$ is flat \cite[Exercise~III.6.10]{h}.
\end{proof}

Note that the regular hypothesis is necessary in general.
In \cite{Ku}, it was shown that in characteristic $p$,
 a local ring $R$ is regular if and only if
$R$ is reduced and flat over $R^p$, the image of $R$ under the Frobenius
morphism. Thus the Frobenius endomorphism of an integral scheme $X$ is
flat if and only if $X$ is regular.

By definition, a morphism $f$ is \emph{faithfully flat}
if $f$ is surjective and flat. We then have the following by
Lemmas~\ref{lem:  int proj, finite=>surj}
and \ref{lem:  normal+finite=>flat}.

\begin{corollary}\label{cor:  normal, int, finite=> faithfully flat}
Let $X$ be a regular, integral scheme.  
Then any finite morphism $f:X\to X$ is faithfully flat.\qed
\end{corollary}

\section{Bimodule algebras}

%\subsection{The algebra}

In this section we examine bimodule algebras in the sense of
\cite{vdb}, while verifying some basic properties that were omitted.

Recall the following definitions from \cite[Definitions~2.1 and 2.3]{vdb}.
\begin{definition}
Given $f:Y\to X$ a morphism of finite-type between noetherian schemes 
and $\mc{M}$ a quasicoherent $\Oh_Y$-module,
we say that $\mc{M}$ is \emph{relatively locally finite} (rlf) for $f$ 
if for all coherent $\mc{M}'\subset \mc{M}$
the restriction $f|_{\supp \mc{M}'}$ is finite.

Given $X$ and $Y$ noetherian $S$-schemes of finite-type, 
an \emph{$\Oh_S$-central $\Oh_X-\Oh_Y$ bimodule} is a 
quasi-coherent $\Oh_{X\times_S Y}$-module, relatively locally finite for the projections 
$\pr_{1,2}:X\times_S Y\to X,Y$.
\end{definition}

\begin{definition}\label{def:  the algebra}
Let $k$ be field.  Let $\mc{L}$ be a quasi-coherent sheaf on a noetherian scheme $X$ 
over $S=\Spec k$ and $\sigma$
a finite $S$-homomorphism $X\to X$.  Let $\mc{L}_0=\mc{O}_X$.  Then
for all integers $j\geq 1$, define 
\[
\mc{L}_j:=\mc{L}\otimes_{\mc{O}_X} \sigma^* \mc{L}\otimes_{\mc{O}_X}\cdots \otimes_{\mc{O}_X}
 (\sigma^{j-1})^*\mc{L}.
 \]

Define $\mc{B}_0=\Oh_X$ and, for all $j\geq 1$, define $\mc{B}_j$ to be the bimodule 
${}_1 (\mc{L}_j)_{\sigma ^j}$.  Then let $\mc{B}=\bigoplus_{j=0} ^\infty \mc{B}_j$ 
(alternatively, let $\mc{B}_j=0$ 
for all $j<0$ and let $\mc{B}=\bigoplus_{j\in\mathbb{Z}} \mc{B}_j$).  

Then let $B_j=\Gamma(\mc{B}_j)=H^0(X,\mc{B}_j)$ and 
define $B=\bigoplus_{n=0} ^\infty B_j=\Gamma(\mc{B})$.
\end{definition}
As in \cite{av} and \cite{vdb}, this construction produces an 
$\mathbb{N}$-graded $k$-algebra $B$ and a graded bimodule algebra $\mc{B}$,  
as Theorem~\ref{thm:  mc{B} is alg} will show shortly.

\begin{lemma}\label{lem:  O_S central bimodule}
Let $S = \Spec k$ and $V,X,Y$ be proper $S$-schemes %of finite-type
with finite $S$-maps $\alpha:V \to X, \beta: V \to Y$. 
Let $\mc{M}$ be a quasi-coherent $\Oh_V$-module. Then 
${}_\alpha \mc{M}_\beta$ is an $\Oh_S$-central $\Oh_X-\Oh_Y$-bimodule.

In particular,
for all $j\in \mb{Z}$, the $\Oh_X$-bimodule $\mc{B}_j$ 
is an $\Oh_S$-central $\Oh_X-\Oh_X$ bimodule.
\end{lemma}

\begin{proof}
Since $X \to S, Y \to S$ are proper, so is $X \times_S Y \to S\times_S S \cong S$
\cite[Corollary~II.4.8(d)]{h}.
Since $\alpha, \beta$ are finite, so is $(\alpha, \beta)$,
as finite maps satisfy the properties of \cite[Exercise~II.4.8]{h}. 
Since every finite map
is proper \cite[Exercise~II.4.1]{h}, we have that
$(\alpha, \beta)$ is closed. Hence $W:=(\alpha, \beta)(V)$ is
a closed subscheme of $X \times_S Y$ and hence proper
\cite[Corollary~II.4.8(a)]{h}.

So consider the commutative diagram
\begin{center}
%\begin{tikzcd}
%V \arrow[r, "\alpha\times  \beta"] \arrow[rd, "\alpha"] & W  \arrow[d, "\pr_1\vert_W"] \\
%& X.
%\end{tikzcd}
\includegraphics[scale=0.43]{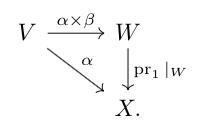}
\end{center}
Since $W$ and $X$ are proper over $\Spec k$, the map $\pr_1\vert_W$ is proper by
\cite[Corollary~II.4.8(e)]{h}. Since $(\alpha, \beta)$ surjects onto $W$
and $\alpha$ is finite, we have that $\pr_1\vert_W$ is quasi-finite
(that is, $\pr_1\vert_W$ has finite fibers). Since
$\pr_1\vert_W$ is proper and quasi-finite, 
the map is finite \cite[$\mathrm{IV}_3$, Th\'{e}or\`{e}me 8.11.1]{ega}.

Now let $\mc{M}$ be a quasi-coherent $\Oh_V$-module
and  $\mc{F}$ be a coherent subsheaf of $(\alpha, \beta)_* \mc{M}$.
Then $\supp \mc{F}$ is a closed subscheme of $W$ \cite[Caution~II.5.8.1]{h}.
So we have $\pr_1\vert_{\supp \mc{F}}$ is finite. 
Thus $(\alpha, \beta)_* \mc{M}$ is rlf for $\pr_1$ and similarly for $\pr_2$.
And thus $(\alpha, \beta)_* \mc{M}$ is an $\Oh_S$-central $\Oh_X-\Oh_Y$ bimodule
by definition.

The claim regarding  $\mc{B}_j$ then follows as a special case.
\end{proof}

\begin{corollary}
Let $X$ be a $k$-scheme of finite-type; recall 
Definition~\ref{def:  the algebra}.   
For any $j\in \mb{Z}$ and any $\Oh_X$-module $\mc{M}$, the bimodule tensor 
product $\; _1\mc{M}_1 \otimes_{\Oh_X}\mc{B}_j$ is an
$\Oh_S$-central $\Oh_X-\Oh_X$ bimodule where $S=\Spec k$.
\end{corollary}

%This corollary will show that the graded pieces of the graded
%bimodule $\mc{B}(\ell) = \Oh_X(\ell)\otimes_{\Oh_X}\mc{B}$ are
%also $\Oh_S$-central $\Oh_X-\Oh_X$ bimodules.  

\begin{proof}  
By Lemma~\ref{lem:  O_S central bimodule} we have that
${}_1 \mc{M}_1$ and $\mc{B}_j$ are both $\Oh_S$-central $\Oh_X-\Oh_X$-bimodules.
By \cite[p.~442]{vdb},  so
is $\; _1\mc{M}_1 \otimes_{\Oh_X}\mc{B}_j$. 
\end{proof}

%\todo{In later arguments, I use that $B_0(X)=\Oh_X(X)=k$ and that in
%general $\mc{B}_n(X)$ is finite dimensional.  
%Make sure it's OK for $X$ irreducible and $\mc{B}_n$ irreducible,
%respectively---and that these assumptions are made.}

Technically each $\mc{B}_j$ is defined on $X\times_{\Spec k} X$; 
since $\Spec k$ is a single point,
one may identify $X\times_{\Spec k}X$ with the Cartesian product $X\times X$.  
%Also, the notation
%$\mc{B}_j(U)$ will be used to mean $\mc{B}_j(X\times U)$, similar to \cite{vdb} 
%given that we will be
%looking at right modules.  

Looking specifically at the right-hand and left-hand module structure, we note
that 
on the right, we have
\begin{multline*}
\mc{B}_n(X\times U)=(1,\sigma ^n)_*\mc{L}_n(X\times U) \\
=\mc{L}_n((1,\sigma^n)^{-1}(X\times U))=\mc{L}_n((\sigma ^n)^{-1}U)=(\sigma^n)_*\mc{L}_n(U).
\end{multline*}
Similarly, on the left, we have 
\[
\mc{B}_n(U\times X)
=(1,\sigma ^n)_*\mc{L}_n(U\times X)
=\mc{L}_n((1,\sigma^n)^{-1}(U\times X))=\mc{L}_n(U).
\]
%Therefore at the level of stalks, as a left module
%$\;_x(\mc{B}_j)\cong \;_x(\mc{L}_j)$, while as a right module
%$(\mc{B}_j)_x\cong((\sigma^j)_*\mc{L}_j)_x$ at any $x\in X$.  (Here we
%have put the subscript for the stalk on the side corresponding to which
%module structure we are using.)  We will discuss stalks of the bimodule
%later and only at generic points.  

So the left module structure in this case is always one from an invertible 
sheaf.  However, it is 
possible in general for $\mc{B}_j$ to have an invertible left module
structure, but not an invertible right module structure 
since in general the pushforward of an invertible sheaf is not invertible.  
In fact, $\mc{B}_j$ is an invertible bimodule if and only if
$\sigma$ is an automorphism \cite[Proposition~2.15]{av}.

On the other hand, by \cite[Proposition~4.5]{h2} if $\sigma$ 
is a faithfully flat morphism and $\mc{F}$ a locally free
sheaf of finite rank, then $\sigma_*\mc{F}$ is also locally free of finite rank. 
Thus these bimodules still have nice tensor properties,
as in Lemma~\ref{lem:  tensor with mc{B}_j is exact}.

%Even though the sheaf $\mc{L}_n$ is formally defined on $X$, 
%it is more intuitively helpful 
%to see $\mc{L}_n$ as defined on the graph of $\sigma^n$.  While the graph of
%$\sigma^n$ is isomorphic to $X$, the set $\{(x,\sigma^n(x))\}$ gives
%a similar two-sided impression as the bimodule. 

\begin{theorem}\label{thm:  mc{B} is alg}
Let $\mc{L}$ be a quasi-coherent sheaf on a noetherian scheme $X$ of 
finite-type and $\sigma$
a finite homomorphism $X\to X$.  Then 
\begin{itemize}
\item the $\mc{B}$ from Definition~\ref{def:  the algebra}
is a graded bimodule algebra; 
\item the $B$ from Definition~\ref{def:  the algebra} is a $\mathbb{N}$-graded $k$-algebra.
\end{itemize}
\end{theorem}

\begin{proof}
The set $B$ automatically has a well-defined addition.  
Having a (compatible) well-defined multiplication 
will follow immediately from the product map $\mu:\mc{B}\otimes_{\mc{O}_X} \mc{B}$ 
that is part of $\mc{B}$'s 
bimodule-algebra structure.

The composition of inclusions $\mc{O}_X\into \mc{B}_0\into \mc{B}$ gives the unit map.  
So all that is left to
verify from \cite[Definition~3.1(1)]{vdb} is the product map.   

The key step in proving this is Lemma~\ref{lem:  welldef mult}, which shows that 
\[
{}_1 \mc{L}_f \otimes {}_1 \mc{M}_g \cong \; {}_1 (\mc{L}\otimes f^* \mc{M} )_{gf}
\]
for any finite homomorphisms $f$ and $g$ from $X$ to itself
and any quasi-coherent sheaves $\mc{L}$ and $\mc{M}$.
Then
\begin{eqnarray*}
{}_1 ( \mc{L}_j)_{\sigma^j} \otimes_{\mc{O}_X} {}_1 (\mc{L}_n)_{\sigma^n}
 &\cong & {}_1 (\mc{L}_j \otimes_{\mc{O}_X} (\sigma^j)^*\mc{L}_n)_{\sigma^{j+n}}\\
&\cong & {}_1 (\mc{L}_{j+n})_{\sigma^{j+n}}\\
\end{eqnarray*}
which is $\mc{B}_j\otimes \mc{B}_n \cong \mc{B}_{j+n}$.
For the ring $B$, that its multiplication is defined and has the usual 
compatibilities follows from the induced map on global sections 
$\mc{B}_j(X)\otimes_{\Oh_X(X)}\mc{B}_n(X)\to \mc{B}_{j+n}(X)$.
\end{proof}

We now examine a special case of \cite[Lemma~2.8(2)]{vdb}.

\begin{lemma}\label{lem:  welldef mult}
Let $X,Y,Z$ be noetherian schemes and $\beta:X\to Y, \delta: Y \to Z$ finite maps.  Let
$\mc{L}$ (respectively $\mc{M}$) be quasicoherent $\Oh_X$-module 
(respectively $\Oh_Y$-module).  Then
\[
{}_{\id_X} \mc{L}_\beta \otimes_{\Oh_Y}  {}_{\id_Y} \mc{M}_\delta 
\cong {}_{\id_X} (\mc{L}\otimes_{\Oh_X} \beta^* \mc{M} )_{\delta \beta}
\]
as $\Oh_S$-central $\Oh_X-\Oh_Z$-bimodules.
Therefore, using the notation from Definition~\ref{def:  the algebra},
we have that $\mc{B}_j\otimes_{\Oh_X}\mc{B}_\ell \cong \mc{B}_{j+\ell}$
for all nonnegative integers $j$ and $\ell$.
\end{lemma}

\begin{proof}
The first claim is \cite[Lemma~2.8(2)]{vdb} in the special case of 
$X=V, Y=W, \alpha = \id_X, \gamma = \id_Y$. 
To follow that proof, we let $p,q: X\times_Y Y \to X, Y$ be 
the two projections. Because $\gamma = \id_Y$, we have that
$p:X \times_Y Y \to X$ is an isomorphism, induced by the
ring isomorphism $B \to B \otimes_A A$. We then have the following commutative diagram,
as in the proof of \cite[Lemma~2.8(2)]{vdb}.
\begin{center}
%\begin{tikzcd}
%& & X \arrow[d, "p^{-1}"] 
%  \arrow[dddll, bend right, "\alpha''" above left] \arrow[dddrr, bend left, "\beta''" above right]\\
%& & X \times_Y Y \arrow[dl, "p" above left] \arrow[dr, "q" above right] \\
%& X \arrow[dl, "\id" above left] \arrow[dr, "\beta" above right]  
%      & & Y \arrow[dl, "\id" above left] \arrow[dr, "\delta" above right] \\
%X & & Y & & Z
%\end{tikzcd}
\includegraphics[scale=0.43]{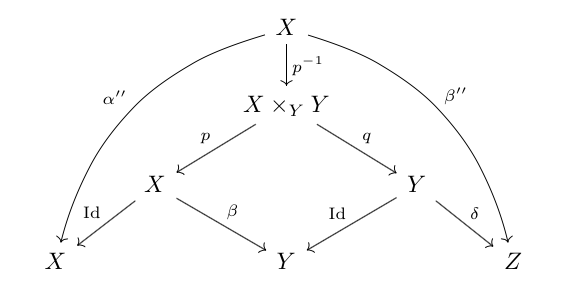}
\end{center}
Then \cite[Lemma~2.8(2)]{vdb} gives
\[
{}_{\id_X} \mc{L}_\beta \otimes_{\Oh_Y} {}_{\id_Y} \mc{M}_\delta 
\cong {}_p (p^*\mc{L} \otimes_{\Oh_{X \times_Y Y}} q^*\mc{M})_{\delta \circ q}.
\]

Note that $\alpha'' = p \circ p^{-1} = \id_X$
and $\beta'' = \delta \circ q \circ p^{-1}$. But the diagram above
also shows $q  = \beta \circ p$. Thus $\beta'' = \delta \circ \beta$.
Taking $\epsilon = p^{-1}$, we are in the special case discussed below 
\cite[Equation~2.2]{vdb}. We thus have
\begin{eqnarray*}
{}_p (p^*\mc{L} \otimes_{\Oh_{X \times_Y Y}} q^*\mc{M})_{\delta \circ q}
&\cong {}_{\id_X} ((p^{-1})^* p^* \mc{L} \otimes_{(p^{-1})^* {\Oh_{X \times_Y Y}}}
  (p^{-1})^* q^* \mc{M})_{\delta \circ \beta} \\
&\cong 
 {}_{\id_X} (\mc{L}\otimes_{\Oh_X} \beta^* \mc{M})_{\delta \beta}
\end{eqnarray*}
as desired.

The final claim follows when $X=Y=Z$.
\end{proof}

\section{Ampleness and noetherianity}

In this section we focus on cohomological vanishing related to the bimodules
$\mc{B}_n = {}_1 (\mc{L}_n)_{\sigma^n}$ where $\mc{L}$ is an invertible $\Oh_X$-module.
We begin with defining ampleness.
Following \cite[Definition~5.1]{vdb}, 
%vdB's \textit{Sklyanin algebras} papers, 
we
have the following definition, with the introduction
of the modifiers \textit{left} and \textit{right}.
We also omit the requirement that the sheaves are eventually generated 
by global sections since that is implied when $\sigma$ is faithfully flat; see
Proposition~\ref{prop:ample sequence implies ample sheaf}.

\begin{definition}\label{def:  ample}
A sequence $(\mc{B}_n)_n$ of coherent $\mc{O}_X$ bimodules is
\emph{left ample} if for any coherent $\mc{O}_X$-module, one has
that for $n\gg 0$, 
%$\mc{B}_n\otimes_{\mc{O}_X}\mc{M}$ is generated
%by global sections and 
$H^q(X,\mc{B}_n\otimes_{\mc{O}_X}\mc{M})=0$ for 
$q>0$.

A sequence $(\mc{B}_n)_n$ of coherent $\mc{O}_X$ bimodules is
\emph{right ample} if for any coherent $\mc{O}_X$ module, one has
that for $n\gg 0$, 
%$\mc{M}\otimes_{\mc{O}_X} \mc{B}_n$ is generated
%by global sections and 
$H^q(X,\mc{M}\otimes_{\mc{O}_X}\mc{B}_n)=0$ for 
$q>0$.
\end{definition}

\begin{definition}\label{def: sigma-ample}
Let $\sigma$ be a map from $X$ to $X$; let $\mc{L}$ be a quasi-coherent sheaf.  
Say that \mc{L} is \emph{left $\sigma$-ample} when the sequence
$\mc{B}_n$, as defined in Definition~\ref{def:  ample}, is left ample; say that
\mc{L} is \emph{right $\sigma$-ample} when that sequence $\mc{B}_n$ is
right ample.
\end{definition}

Since ampleness depends on tensor products, it is useful to know that
tensoring is exact.

\begin{lemma}\label{lem:  tensor with mc{B}_j is exact}
Let $X$ be an integral scheme with locally free sheaf
$\mc{L}$. Let $\alpha, \beta$ be finite faithfully flat endomorphisms of $X$.
Then tensoring the $\Oh_X$-biomdule ${}_\alpha(\mc{L})_\beta$ with 
quasicoherent $\Oh_X$-modules is exact.

Thus with $\mc{B}_j$ as in Definition~\ref{def:  the algebra},
we have that tensoring $\Oh_X$-modules with $\mc{B}$ is exact.
\end{lemma}
\begin{proof}
By \cite[Proposition~2.6]{vdb}, tensor products of bimodules is
right exact, but it will be straightforward to prove exactness directly.
Consider the exact sequence of quasicoherent $\Oh_X$-modules
\[
\mc{M}_1 \to \mc{M}_2 \to \mc{M}_3.
\]
By \cite[Lemma~2.8(1)]{vdb}, we have
\[
\mc{M}_i \otimes_{\Oh_X} {}_\alpha(\mc{L})_\beta
\cong {}_\alpha(\alpha^*\mc{M}_i \otimes_{\Oh_X} \mc{L})_\beta.
\]
Since $\alpha$ is flat, pullback by $\alpha$ is exact 
\cite[$\mathrm{IV_2}$, Th\'{e}or\`{e}me~2.1.3]{ega}. Since $\mc{L}$ is locally free,
tensoring with $\mc{L}$ is exact. Finally, since $(\alpha,\beta)$ is a finite
map, pushfoward by $(\alpha,\beta)$ is exact. 

Since $\mc{B}_j$ is of the form ${}_\alpha(\mc{L})_\beta$ 
and $\mc{B}$ is a direct sum of the $\mc{B}_j$, we have that
tensoring with $\mc{B}$ is exact \cite[Proposition~I.10.1]{stn}.
\end{proof}

When considering the ampleness of the sequence $\mc{B}_n$, it will
be convenient to examine the related $\Oh_X$-modules
$\mc{F} \otimes \mc{L}_n$ and $\mc{L}_n \otimes (\sigma*)^n \mc{F}$, as
these have the same cohomology as $\mc{F} \otimes \mc{B}_n$ and
$\mc{B}_n \otimes \mc{F}$ since $(1, \sigma^n)$ is finite.

Also, when $f:X \to X$ is finite and faithfully flat, and $\mc{E}$ is a locally free
coherent sheaf, we have that $f_*\mc{E}$ is a locally free coherent sheaf
\cite[Proposition~4.5]{h2}. Hence tensoring with $f_*\mc{E}$ is exact.
Further, if $\mc{F}$ is any coherent sheaf, then we have the projection formula
\begin{equation}\label{eq:projection formula}
f_*(\mc{E}) \otimes \mc{F} \cong f_*(\mc{E} \otimes f^*\mc{F}).
\end{equation}
This holds for any coherent $\mc{F}$ because $f$ is finite
\cite[Lemma~5.7]{AraK}.
Thus we may study the left ampleness of $\mc{B}_n$ by tensoring
with the $\Oh_X$-modules $\sigma^n_*(\mc{L}_n)$, allowing us
to ignore the bimodule structure. Therefore we may apply the
following proposition.

\begin{proposition}\label{prop:ampleness equivalences} 
\cite[Theorem~7.2, Proposition~7.3]{k2}
Let $X$ be a projective scheme with $\mc{E}_n$ a sequence of locally free
coherent sheaves. Consider the following properties.
\begin{enumerate}
\item\label{F-ample1} For any coherent sheaf $\mc{F}$, there exists $n_0$ such that
$H^q(\mc{F} \otimes \mc{E}_n)=0$ for $q>0, n \geq n_0$.
\item\label{p-ample1} For any coherent sheaf $\mc{F}$, there exists $n_0$ such that
$\mc{F} \otimes \mc{E}_n$ is generated by global sections for $n \geq n_0$.
\item\label{ample} For any invertible sheaf $\mc{H}$, there exists $n_0$ such
that $\mc{H} \otimes \mc{E}_n$ is an ample locally free sheaf for $n \geq n_0$.
\end{enumerate}
Then (\ref{F-ample1})  $\implies$ (\ref{p-ample1}) 
 $\implies$ (\ref{ample}).

If the $\mc{E}_n$ are invertible, then the conditions
are equivalent. \qed
\end{proposition}

Note that the conditions above are not equivalent for general locally free
sheaves \cite[Remark~7.4]{k2}.

In our specific situation, we obtain the following.

\begin{proposition}\label{prop:ample sequence implies ample sheaf}
Let $X$ be a projective integral scheme with finite faithfully flat
endomorphism $\sigma$. Define $\mc{B}_n$ and $\mc{L}_n$ as in 
Definition~\ref{def:  the algebra} with $\mc{L}$ invertible. 
Let $\mc{F}$ be a coherent sheaf and let $\mc{H}$ be an invertible sheaf.
Then
\begin{enumerate}
\item\label{claim1} If $\mc{B}_n$ is right ample, then for all $n \gg 0$,
$\mc{F} \otimes \mc{L}_n$ is generated by global sections and
$\mc{H} \otimes \mc{L}_n$ is an ample invertible sheaf.
\item\label{claim2} If $\mc{B}_n$ is left ample, then for all $n \gg 0$,
$\mc{L}_n \otimes (\sigma^n)^*\mc{F}$ is generated by global sections and
$\mc{L}_n \otimes (\sigma^n)^*\mc{H}$ is an ample invertible sheaf.
\end{enumerate}
\end{proposition}
\begin{proof}
The first claim follows immediately from Proposition~\ref{prop:ampleness equivalences}.

For the second claim, consider the sequence of coherent sheaves
$\mc{E}_n = \sigma^n_* (\mc{L}_n)$, which are locally free \cite[Proposition~4.5]{h2}. 
Then $\mc{E}_n$ satisfies \eqref{F-ample1}
of Proposition~\ref{prop:ampleness equivalences}
because of the isomorphism \eqref{eq:projection formula}
and the fact that push forward by a finite morphism preserves cohomology
\cite[Exercise~III.4.1]{h}.

Therefore $\sigma^n_* (\mc{L}_n \otimes (\sigma^n)^*\mc{F})$ is generated
by global sections
and $\sigma^n_*(\mc{L}_n \otimes (\sigma^n)^*\mc{H})$ is an ample
locally free sheaf for $n \gg 0$ by Proposition~\ref{prop:ampleness equivalences}. 
As in the proof of \cite[Proposition~4.5]{h2},
we can pull back 
\[
\oplus \Oh_x \to \sigma^n_* (\mc{L}_n \otimes (\sigma^n)^*\mc{F}) \to 0
\]
by $\sigma^n$ to get that
\[
(\sigma^n)^*\sigma^n_* (\mc{L}_n \otimes (\sigma^n)^*\mc{F})
\]
is generated by global sections. But since $\sigma^n$ is affine, the natural
map
\[
(\sigma^n)^*\sigma^n_* (\mc{L}_n \otimes (\sigma^n)^*\mc{F}) 
\to \mc{L}_n \otimes (\sigma^n)^*\mc{F}
\]
is surjective. Hence $\mc{L}_n \otimes (\sigma^n)^*\mc{F}$ is generated by
global sections.

Since  $\sigma^n_*(\mc{L}_n \otimes (\sigma^n)^*\mc{H})$ is an ample
locally free sheaf for $n \gg 0$, we have that
$\mc{L}_n \otimes (\sigma^n)^*\mc{H}$ is an ample invertible sheaf by
\cite[Proposition~4.5]{h2}.
\end{proof}

Due to Proposition~\ref{prop:ample sequence implies ample sheaf}, we will examine the ampleness
of the invertible sheaves $\mc{H} \otimes \mc{L}_n$ and
$\mc{L}_n \otimes (\sigma^n)^* \mc{H}$ for arbitrary
invertible sheaves $\mc{H}$. As in \cite{k}, the key idea is to 
examine the action of $\sigma^*$ on the numerical equivalence classes
of invertible sheaves, or equivalently Cartier divisors.

Let $\Num(X)=A^1_{\Num}(X)$ be
the set of Cartier divisors of
$X$, modulo numerical equivalence. We follow the work of
\cite[Sections~3-4]{k}.
Recall that $\Num(X) \cong \ZZ^\ell$, for some $\ell$, as a group and so
we can take
$P \in M_\ell(\ZZ)$ where $P$ is the action of $\sigma^*$ on $\Num(X)$
\cite[p.~305, Remark~3]{kl}.

\begin{lemma}\label{lem: eigenvalue greater than 1}
Let $X$ be a projective integral scheme
with $\sigma:X \to X$ finite. Let $P$
be the action of $\sigma^*$ on $\Num(X)$. Then 
the spectral radius $r$ of $P$
is an eigenvalue of $P$ and $r \geq 1$.
If the largest Jordan block containing $r$ is $m \times m$
and $\lambda$ is an eigenvalue with $|\lambda|=r$, then
any Jordan block containing $\lambda$ has size
less than or equal to $m \times m$.

If $X$ is normal and $\sigma$ is not an automorphism, then $r > 1$.
\end{lemma}
\begin{proof} 
We have that $\sigma$ is surjective by Lemma~\ref{lem:  int proj, finite=>surj}.
So $\sigma^*$ is an automorphism of $\Num(X)$ \cite[p.~331, Remark~1]{kl}.
Thus $P$ is invertible as a matrix (over $\QQ$).

Let $\ell$ be the rank of $\Num(X)$. Since 
$P \in M_\ell(\ZZ)$, we have that $\det(P) \in \ZZ$
and $|\det(P)| \geq 1$ since $P$ is invertible over $\QQ$.
Thus we have that the spectral radius $r \geq 1$.

Since $\sigma$ is surjective, $\sigma^*$ preserves the cone of nef divisors
in $\RR^\ell$
\cite[p.~303, Proposition~1]{kl}
and that cone has non-empty interior (the ample cone)
\cite[p.~325, Theorem~1]{kl}. 
Thus $r$ is an eigenvalue of $P$ and we also have
the claim regarding Jordan blocks \cite[Theorem~3.1]{vandergraft}.

Now suppose $X$ is normal and $\sigma$ is not an automorphism.
Then $\deg \sigma > 1$ by Lemma~\ref{lem:  normal, int, finite deg 1=>auto}.

Let $D$ be an ample divisor, or more generally a divisor
with self-intersection number $(D^{\dim X}) > 0$. 
Consider the integer valued function 
\[
g(m)=((\sigma^{*m}D)^{\dim X}) = ((P^m D)^{\dim X})
\] with $m \in \NN$.
By \cite[p.~299, Lemma~2, Proposition~6]{kl}, 
\[
((P^m D)^{\dim X}) = (\deg \sigma^m)(D^{\dim X})
= (\deg \sigma)^m (D^{\dim X}).
\] Since $\deg \sigma > 1$, we have
that $g$ grows exponentially.

Suppose for contradiction that $r=1$.
Then $\det(P) = \pm 1$ and
every eigenvalue has absolute value $1$. 
This implies that every eigenvalue is a root of unity
since $P \in \GL_\ell(\ZZ)$ \cite[Lemma~5.3]{av}.
That is, we say that $P$ is quasi-unipotent.
We may replace $P$ with a power of $P$ and assume
$P=I+N$ where $I$ is the identity matrix and
$N$ is a nilpotent matrix.
And so $g$ has polynomial growth 
since $P^mD = (I+N)^mD$ can be expressed as a 
``polynomial'' in $m$ with divisors as coefficients,
as in \cite[Equation~4.2]{k}.
Thus the self-intersection number is a polynomial with
integer coefficients.

But $g$ grows exponentially, so it must be that $r \neq 1$.
Thus $r > 1$. 
\end{proof}

Note that it is possible for an automorphism to yield
a spectral radius $r > 1$ \cite[Example~3.9]{k}.

The spectral radius $r$ is key to determining the ampleness of 
the appropriate sequences. When $\sigma$ is an automorphism,
we have an exact characterization.

\begin{proposition} \cite[Theorem~4.7, Corollary~5.1]{k}
Let $X$ be a projective integral scheme with automorphism
$\sigma$ and invertible sheaf $\mc{L}$. Define $\mc{B}_n, \mc{L}_n$
as in Definition~\ref{def:  the algebra}. Then $\mc{B}_n$
is right ample if and only if $\mc{B}_n$ is left ample.
This occurs if and only if $\mc{L}_n$ is an ample
invertible sheaf for $n \gg 0$ and $r=1$ with $r$ the spectral 
radius of the numerical action of $\sigma^*$. \qed
\end{proposition}

This leaves us to consider the case of $r > 1$. To do so, we 
determine the growth of intersection numbers.

\begin{lemma}\label{lem: P^mD growth}
Let $X$ be a projective integral scheme with finite endomorphism $\sigma$. 
Let $D$ be a divisor. Let $P$
be the action of $\sigma^*$ on $\Num(X)$ with spectral radius $r$ and
largest Jordan block associated to $r$ of size $(j+1)\times (j+1)$.

Then
\begin{enumerate}
\item \label{lem: P^mD growth1} 
For all curves $C$, there exists $c_1 > 0$ such that
\[
(P^m D . C) \leq c_1 m^j r^m \text{ for all } m > 0.
\]
\item If $D$ is ample, then there exists a curve $C$ and $c_2 > 0$ such that
\[
(P^m D . C) \geq c_2 m^j r^m \text{ for all } m > 0.
\]
\end{enumerate}
\end{lemma}
\begin{proof}
Let $v_1, v_2, \dots, v_\rho$ be a basis for $\Num(X) \otimes \CC$
associated to a Jordan canonical form of $P$, with an $(j+1) \times (j+1)$
block for $r$ in the upper left. Since $r$ is real, the vectors
$v_1, v_2, \dots, v_{j+1}$ have real components.

Applying $P^m$ to $v_{i+1}$ for $i=0, \dots, j$ gives
\begin{equation}\label{eq: P^m growth}
P^m v_{i+1} = \binom{m}{i} r^{m-i} v_1 + \binom{m}{i-1} r^{m-i+1}v_2 + \dots + r^m v_{i+1}.
\end{equation}
Recall that every eigenvalue $\lambda$ of $P$ has $|\lambda| < r$ or
the Jordan blocks of $\lambda$ are no larger than $(j+1)\times (j+1)$
by Lemma~\ref{lem: eigenvalue greater than 1}. 
Thus for all $i =1, \dots, \rho$, the absolute values of the scalar coefficients of 
$P^m v_i$ cannot grow faster than $r^{m-j} \binom{m}{j}$.

Let $D = \sum a_i v_i$. Since $D \in \Num(X)$, we have $a_i \in \RR$
whenever $v_i$ corresponds to a real eigenvalue. If $v_i$ corresponds
to a complex eigenvalue $\lambda$, then the complex conjugate $\overline{\lambda}$
is also an eigenvalue with corresponding vector $\overline{v_i}$
and coefficient $\overline{a_i}$. Letting $v_i' = a_i v_i + \overline{a_i v_i}$ we
write $D$ as the sum of real vectors. Considering $P^mD$, we see that
no scalar coefficient of a $v_i$ or $v_i'$ can grow faster than $r^{m-j} \binom{m}{j}$
by Equation~\eqref{eq: P^m growth}.
Thus intersecting with any curve $C$ we have $(P^m D.C) \leq c_1 m^j r^m$
for some $c_1 > 0$ and all $m > 0$. (We can choose $c_1$ large enough
to handle the case of small $m$.)

Now suppose $D$ is ample. Then $D$ is an element of the interior
of the nef cone \cite[p.~325, Theorem~1]{kl}
and so there exists $\alpha > 0$ such that $\alpha D - v_{j+1}$ is
an element of the ample cone \cite[p.~1209]{vandergraft}.
Since $\sigma$ is finite, $P^m(\alpha D - v_{j+1})$ is ample for all $m \geq 0$
\cite[Proposition~4.3]{h2}.
Thus for all curves $C$ and all $m > 0$, $\alpha(P^m D.C) > (P^m v_{j+1}.C)$
since the intersection number of an ample divisor and any curve is positive.

Since $v_1 \neq 0 \in \Num(X)$, there exists a curve $C$ such that
$(v_1.C) \neq 0$. Replacing $v_i$ with $-v_i$ if necessary, we can assume
$(v_1.C) > 0$. Then by Equation~\eqref{eq: P^m growth}, we have
$\alpha(P^m D. C) > c_3 m^j r^m$ for some $c_3 > 0$ and all $m \gg 0$. 
We can choose $c_3$ small enough to make the inequality true for all $m > 0$.
Then take $c_2 = c_3/\alpha$ to complete the proof.
\end{proof}

While the lemma above holds for $r=1$, we have different behavior
for the intersection numbers of sums of $P^i D$, depending on $r$. For $r=1$,
the intersection numbers $(D+PD+\dots+P^{m-1}D.C)$ grow at most like the
polynomial $m^{j+1}$ \cite[Equation~(4.3)]{k}. However, when $r > 1$,
these numbers grow at most like $m^j r^m$.

\begin{lemma}\label{lem: sum growth}
Use the hypotheses of Lemma~\ref{lem: P^mD growth} with $D$ any divisor. 
Suppose $r > 1$. Then for all curves $C$, there exists $c_3 > 0$ such that
\[
(D + PD + P^2 D + \dots + P^{m-1}D.C) \leq c_3 m^j r^m\text{ for all } m > 0.
\]
\end{lemma}
\begin{proof} Consider $\sum_{i=0}^{m-1} r^i = \frac{r^m-1}{r-1}$.
By differentiating $j$ times with respect to $r$, we see that
$\sum_{i=0}^{m-1} i^j r^i$ is bounded above by $c_4 m^j r^m$ for some
$c_4 > 0$. The claim then follows from
Lemma~\ref{lem: P^mD growth}\eqref{lem: P^mD growth1}.
\end{proof}

We can now prove that left ampleness never occurs when $\sigma$ is not an automorphism
and $X$ is normal.
For simplicity, define
\[
\Delta_m = D + \sigma^* D + (\sigma^*)^2 D + \dots + (\sigma^*)^{m-1}D.
\]

\begin{proposition}\label{prop: non left ample divisors}
Let $X$ be a projective integral scheme with $\sigma: X \to X$ finite.
Let $P$ be the action of $\sigma^*$ on $\Num(X)$ and let $r$ be the
spectral radius of $P$, with $r > 1$.
Let $D$ be a divisor on $X$. Then there exists an ample divisor $H$ such that
$\Delta_m - (\sigma^m)^*H$
is not ample for all $m > 0$.
\end{proposition}
\begin{proof}
Let $H$ be an ample divisor. Choose a curve $C$ and $c_2 > 0$
 such that $(P^m H.C) \geq c_2 m^j r^m$ for all $m > 0$,
  with $j$ as in Lemma~\ref{lem: P^mD growth}.
Then there exists $c_3 > 0$ such that $(\Delta_m . C) \leq c_3 m^j r^m$
for all $m > 0$
by Lemma~\ref{lem: sum growth}.
Replacing $H$ with an integer multiple, we can assume $c_2 > c_3$. 
Thus $(\Delta_m - (\sigma^*)^mH. C) < 0$ for all $m > 0$.
Since ample divisors have positive intersection with every curve, we are done.
\end{proof}

Combining Propositions~\ref{prop:ample sequence implies ample sheaf} and
\ref{prop: non left ample divisors} we immediately have the following. 
Recall that if $X$ is normal and $\sigma$ is not an automorphism, then $r > 1$ by
Lemma~\ref{lem: eigenvalue greater than 1}. 
Note that by Corollary~\ref{cor:  normal, int, finite=> faithfully flat}, the
following holds for any finite endomorphism when $X$ is regular.

\begin{corollary}\label{cor: not left ample}
Let $X$ be a projective scheme with finite faithfully flat endomorphism $\sigma$.
Let $P$ be the action of $\sigma^*$ on $\Num(X)$ and let $r$ be the
spectral radius of $P$, with $r > 1$. Then $\mc{B}_n$ is not left ample,
where $\mc{B}_n$ is as in Definition~\ref{def:  the algebra} with $\mc{L}$ an
invertible $\Oh_X$-module. \qed
\end{corollary}

It is however possible for the sequence $\mc{B}_n$ to be right ample
when $\sigma$ is not an automorphism. This is in contrast to the automorphism case,
where left and right ampleness are equivalent \cite[Corollary~5.1]{k}.
Note that we do not need $\sigma$ to be faithfully flat due to the equivalences of
Proposition~\ref{prop:ampleness equivalences}.

\begin{proposition}\label{prop: ample eigenvector}
Let $X$ be a projective integral scheme with finite endomorphism $\sigma$
and $P$ the action of $\sigma^*$ on $\Num(X)$. Suppose there exists
an ample invertible sheaf $\mc{L} = \Oh_X(D)$ such that $D$ is an eigenvector
of $P$. Then the sequence $\mc{B}_n = {}_1(\mc{L}_n)_{\sigma^n}$ is
right ample.
\end{proposition}

\begin{proof}
Note that due to Lemma~\ref{lem: P^mD growth}, the eigenvalue related to $D$ must
be the spectral radius $r$. 
For any divisor $H$, there exists $\alpha_0 \in \RR$ 
such that $\alpha D+H$ is in the
ample cone for all $\alpha \geq \alpha_0$ \cite[p.~1209]{vandergraft}. 
Thus $\Delta_m + H$ is ample for all $m \gg 0$. 
So by Proposition~\ref{prop:ampleness equivalences}, the sequence $\mc{B}_n$
is right ample.
\end{proof}

Note that since $D \in \ZZ^\ell$, we must have $r \in \QQ$.
Also, we have self-intersection number \[
(\sigma^*D)^{\dim X} = r^{\dim X} (D)^{\dim X} = (\deg \sigma)(D)^{\dim X}
\]
\cite[p.~299, Proposition~6]{kl}.
Since $X$ is integral, $\deg \sigma \in \NN$. Thus $r \in \NN$.
In fact, it is known that when an ample divisor is an eigenvector,
 we can write $P= r (I \oplus O)$ where
$I$ is the $1 \times 1$ identity matrix
 and $O$ is an orthogonal matrix with $\QQ$ coefficients
\cite[p.~330, Proposition~3]{kl}. But we will not need this.

One important example of a finite endomorphism
satisfying Proposition~\ref{prop: ample eigenvector} is the relative
Frobenius endomorphism $f$ for characteristic $p > 0$. 
In this case, $P$ is just the scalar matrix given by multiplication by $p$
\cite[Lemma~2.4]{AraK}.
In the next section we examine the twisted ring obtained from
the Frobenius endomorphism on $\mb{P}^m$.

It is also trivial that Proposition~\ref{prop: ample eigenvector} holds when
$\Pic(X) \cong \ZZ$. 

We end this section by showing that the bimodule algebras obtained from
non-automorphisms are not noetherian.

\begin{theorem}\label{thm: nonnoetherian ring}
Let $X$ be a projective normal integral scheme
with finite endomorphism $\sigma$, with $\sigma$ not an automorphism.
Let $\mc{L}$ be an invertible $\Oh_X$-module and define $\mc{B}_n$
and $B = \oplus \Gamma(\mc{B}_n)$ as in Definition~\ref{def:  the algebra}.
Suppose $\mc{B}_n$ is a right ample sequence. Then $B$ is neither left
nor right noetherian.
\end{theorem}
\begin{proof}
Let $\mc{L} = \Oh_X(D)$ and let $P$ be the action of $\sigma^*$ on $\Num(X)$, with
spectral radius $r$. We have $r > 1$ by Lemma~\ref{lem: eigenvalue greater than 1}.

By Lemma~\ref{lem: P^mD growth}, 
choose a curve $C$ such that $(P^n D.C) \geq cr^n$ for some $c > 0$.
We have an exact sequence of coherent sheaves
\[
0 \to \mc{I} \to \Oh_X \to \Oh_C \to 0.
\]
Tensoring with $\mc{B}_n$ on the right, we have
$H^1(X, \mc{I} \otimes \mc{B}_n)=0$ for $n \gg 0$. 
Thus $\Gamma(\mc{B}_n) \to \Gamma(\Oh_C \otimes \mc{B}_n)$ is surjective
for $n \gg 0$. By the Riemann-Roch formula for curves,
$\Gamma(\Oh_C \otimes \mc{B}_n)$ grows exponentially and hence so
does $\Gamma(\mc{B}_n)$.

Since $B$ has exponential growth, $B$ is neither left nor right noetherian
by \cite[Theorem~0.1]{sz}.
\end{proof}

We also find that the bimodule algebra $\mc{B}$ is not right noetherian.
See \cite{vdb} for details on the category of 
right $\mc{B}$-modules.

\begin{corollary}\label{cor: nonnoetherian bimodule algebra}
Let $X$ be a regular projective integral scheme over an algebraically
closed field $k$. Let $\sigma$ be a finite endomorphism
which is not an automorphism. Define $\mc{B}_n$ and $\mc{B}$ 
as in Definition~\ref{def:  the algebra}.
Suppose $\mc{B}_n$ is a right ample sequence. Then $\mc{B}$ is not
right noetherian in the category of right $\mc{B}$-modules.
\end{corollary}
\begin{proof}
Suppose $\mc{B}$ is right noetherian.
First, note that the $\mc{B}_n$ are left flat by
Lemma~\ref{lem:  tensor with mc{B}_j is exact}.
Thus the category of coherent right $\mc{B}$-modules is abelian
\cite[Corollary~3.8]{vdb}.

Therefore, all hypotheses of \cite[Theorem~5.2]{vdb} hold.
Thus $B = \Gamma(\mc{B})$ is right noetherian. But
this contradicts Theorem~\ref{thm: nonnoetherian ring}.
So $\mc{B}$ cannot be right noetherian.
\end{proof}

\section{Frobenius endomorphism on projective space}\label{sec:  frobenius example}
In this section, we examine more closely a specific case of 
a twisted homogeneous coordinate ring, 
specifically the one created from the relative Frobenius morphism $f$ (the
$k$-algebra homomorphism induced by
$x_i\mapsto x_i ^p$) on $\mathbb{P}^m _k$ where $\Char~k=p>0$ 
and the standard coordinate ring of $\mb{P}^m$ is written
$k[x_0,\ldots, x_m]$.  Note that ``relative'' means that $f$ is the identity on $k$,
so that $f$ is a $k$-morphism. This is a modification of the absolute Frobenius
morphism, in which $a \mapsto a^p$ for $a \in k$, exploiting the fact
that the Frobenius morphism is an isomorphism of a perfect field $k$.

\begin{definition}
Let $k$ be a perfect field of characteristic $p>0$.  
Let $f$ be the relative Frobenius map of $\mathbb{P}^m _k$, 
i.e., the map generated by the Frobenius 
homomorphism described by $x_i\to x_i ^p$ on $k[x_0,x_1,\ldots, x_m]$.
Let $F=\bigoplus F_n$ where 
\[
F_n
=\Gamma(\Oh(1)\otimes f^*\Oh(1)\otimes (f^2)^*\Oh(1)\otimes \cdots \otimes (f^{n-1})^*\Oh(1)).
\]
As with the standard construction, for $a\in F_i$ and $c\in F_j$, 
the multiplication of $a$ and $c$ is given by
$af^{i}(c)$.  

As $(f^j)^*\Oh(1)\cong \Oh(p^j)$, we have that 
$F_n = 
\Gamma(\Oh(1+p+p^2+\cdots +p^{n-1} ))$.
\end{definition}

%\begin{notation}
For convenience, let 
\[
e_n=1+p+p^2+\cdots + p^{n-1} =\frac{p^n-1}{p-1}.
\] 
%The sequence $e_n$ can be described by the recursive relation
%\[
%e_n=\left\{\begin{array}{ll} 1 & \text{ if } n=1\\ 
%                             p^{n-1}+e_{n-1} &\text{ if } n>1\end{array}\right.
%\]
%as well as by the recursive relation
%\[
%e_n=\left\{\begin{array}{ll} 1 & \text{ if } n=1\\ 
%                             pe_{n-1}+1 & \text{ if } n>1\end{array}\right.
%\]
%\end{notation}

In this notation, we have that $\dim_k F_n=\binom{e_n+m}{m}$, which 
 grows like $(p^m)^n$.  Hence, $\dim_k F_n$ grows exponentially with
$n$ and $F$ cannot be either left or right noetherian by 
\cite[Theorem~0.1]{sz},
just as in Theorem~\ref{thm: nonnoetherian ring}.

Moreover, $F$ is not finitely generated as a $k$-algebra, 
except in the case when both
$m=1$ and $p=2$, as we will now show.  The proof breaks into two
cases, one where $p>2$ and one where $p=2$.

Before getting to the proofs, a few words about notation.  Given any
homogeneous element $a$ of $F$, it has the usual total degree as a polynomial 
and the degree referring to which $F_j$ the element $a$ belongs.  We adopt the
following convention for this section, $\deg a$ is the total degree
of $a$ considered as a polynomial in $k[x_0,\ldots, x_m]$ and $\deg_{x_i} a$
is the exponent of $x_i$ when $a$ is written as a polynomial 
in $k[x_0,\ldots, x_m]$.  So $\deg x_0 ^{p+1}=p+1$ even though 
$x_0 ^{p+1}\in F_2$.  

Also, given two elements $a_1$ and $a_2$ in $F$, their
product as elements of $F$ will be denoted $a_1\cdot a_2$ while if
$a_1$ and $a_2$ are considered solely as polynomials, the notation
$a_1a_2$ will denote the standard multiplication of polynomials.
Similarly, given sets of elements $A_1$ and $A_2$, we define 
$A_1\cdot A_2=\{ \sum_{\text{finite}} a_1\cdot a_2|a_i\in A_i\}$ and
$A_1A_2=\{\sum_{\text{finite}} a_1a_2|a_i\in A_i\}$.

\subsection{The case \texorpdfstring{$p>2$}{p>2}}\label{subsec:  Frob, p>2, non fin gen}

First, the case when $p>2$.  To show that $F$ is not finitely generated,
we will show that any element $z$ of $F_n$ such that $\deg_{x_0} z=p^{n-1}-1$
will not be the product of elements of strictly smaller degree for
any $n>1$.  The proof is by contradiction.  So assume that there is 
some $z\in F_n$ with $\deg_{x_0}z=p^{n-1} -1$ such that $z=uv$ where
$u\in F_a$ and $v\in F_b$ for some integers $a$ and $b$ such that
$a+b=n$ and $1<a,b<n$.  (As $z$ is a monomial, $u$ and $v$ are
also monomials.)  Then $\deg_{x_0} u=c$ for some integer 
$c\leq \frac{p^a-1}{p-1}$ and $\deg_{x_0} v =d$ for some integer 
$d\leq \frac{p^b-1}{p-1}$.  By definition of the multiplication, we now
have $\deg_{x_0}(uv)=c+p^a d$.  Since $z=uv$, this means
$c+p^a d = p^{n-1}-1$.  As $a\leq n-1$, we have $c\equiv -1\pmod{p^a}$.
Since $0\leq c$, this means $c\geq p^a-1$, but as $c \leq \frac{p^a-1}{p-1}$
this is impossible for $p>2$.  Hence $z$ cannot be generated by elements
of smaller degree.

\subsection{The case \texorpdfstring{$p=2$}{p=2}}\label{subsec: Frob example, p=2}
When $m=1$, the algebra $F$ is in fact generated by $F_1$ as follows.
In this case 
\[
F_n=\Oh(\frac{2^n-1}{2-1})(X)=\Oh(2^n-1)(X)=\bigoplus_{j=0} ^{2^n-1} kx^{2^n-1-j}y^j.
\]
For any $z\in F_n$, we have $x\cdot z=xz^2$ and $y\cdot z=yz^2$ (where
the right-hand side is written as an element of $F_{n+1}$).  
Therefore $x\cdot F_n=\bigoplus_{j=0} ^{2^n-1}kx^{2^{n+1}-1-2j}y^{2j}$
and $y\cdot F_n=\bigoplus_{j=0} ^{2^n-1} kx^{2^{n+1}-2-2j}y^{2j+1}$ 
Hence, $F_1\cdot F_n=x\cdot F_n+y\cdot F_n=F_{n+1}$ and, by induction, 
$F_\ell$ must
be generated by $F_1$ for all $\ell\geq 1$.

%The details are in the next section, which discusses the special case
%when $m=1$ for any positive characteristic.  

However, when $m\geq 2$, the algebra $F$ is not finitely generated.
In this case, we will show that the element 
\[
z=x_0 ^{2^{n-1}-1} x_1 ^{2i+1} x_2 ^{2j+1} \in F_n
\]
is not generated by elements of smaller degree for any nonnegative 
integers $i$ and $j$ (here $2i+1+2j+1=2^{n-1}$).  As before, the proof
proceeds by contradiction.  Suppose that 
$z=uv$ for some $u\in F_r$ and $v\in F_s$ where $r+s=n$ and $1<r,s<n$.  
Since $z$ is a monomial only
in $x_0, x_1$ and $x_2$, so are $u$ and $v$.  Write $u=x_0 ^a x_1 ^b x_2 ^c$
and $v=x_0 ^d x_1 ^e x_2 ^f$; here $a+b+c=2^r-1$ and $d+e+f=2^s-1$.  
By definition of the multiplication
$uv = x_0 ^{2^r d +a} x_1 ^{2^r e +b} x_2 ^{2^r f+c}$.  Since $z=uv$,
this means $2^r d+a=2^{n-1}-1$, $b+2^r e = 2i+1$ and $c+2^r g =2j+1$.  
Then as $r\leq n-1$, this means $a\equiv -1\pmod{2^r}$.
Thus $a=2^r g-1$ for some positive integer $g$.
Then $2^r -1 = a+b+c=2^r g-1 +b+c$; the only way this is possible is
for $b=c=0$ and $a=2^r-1$.  Now, the other two equations become
$2^r e = 2i+1$ and $2^r f =2j+1$, neither of which is possible.   
Hence $z$ cannot be generated by elements of smaller degree.

Summing up our work, we have the following.
\begin{proposition}\label{prop: Frobenius example}
Let $X=\mb{P}^m$ over a perfect field of characteristic $p > 0$. 
Let $f$ be the relative Frobenius endomorphism. 
Then the twisted homogeneous coordinate ring 
$F = \oplus \Gamma\left(\Oh\left(\frac{p^n-1}{p-1}\right)\right)$
is not finitely generated, unless $p=2$ and $m=1$. In that case, $F$ is generated
by $F_1 = \Gamma(\Oh(1))$, but $F$ is not noetherian. \qed
\end{proposition}

\section*{References}

%\bibliography{mybibfile}

%\section{References}

\end{document}